\documentclass[11pt]{article}

\usepackage{amsmath,amsthm,amstext,amscd,amssymb,euscript,mathrsfs}
\usepackage{graphics,color,calc,graphicx}
\usepackage[utf8x]{inputenc}
\usepackage{calrsfs}
\usepackage{hyperref}

\newcommand{\cC}{\mathcal{C}}

\newcommand{\E}{\mathbb{E}}
\newcommand{\T}{\mathbb{T}}
\newcommand{\cA}{\mathcal{A}}

\newcommand{\fF}{\mathfrak{F}}

\newcommand{\Zd}{\mathbb Z^d}

\newcommand{\WSF}{\mathsf{WSF}}
\renewcommand{\phi}{\varphi}

\newcommand{\caC}{{\mathcal C}}

\newcommand{\caR}{{\mathcal R}}

\theoremstyle{plain}
\newtheorem{theorem}{Theorem}
\newtheorem{lemma}{Lemma}
\newtheorem{proposition}{Proposition}

\theoremstyle{definition}

\theoremstyle{remark}
\newtheorem{remark}{Remark}

\newcommand{\eqn}[2]{\begin{equation}\label{#1}#2\end{equation}}
\newcommand{\eqnst}[1]{\begin{equation*}#1\end{equation*}}
\newcommand{\eqnspl}[2]{\begin{equation}\begin{split}\label{#1}%
    #2\end{split}\end{equation}}
\newcommand{\eqnsplst}[1]{\begin{equation*}\begin{split}%
    #1\end{split}\end{equation*}}

\def\es{\emptyset}
\def\eps{{\varepsilon}}
\def\Zd{{\mathbb{Z}^d}}
\def\conn{{\leftrightarrow}}

\def\Av{{\mathrm{Av}}}
\def\bbT{{\mathbb{T}}}
\def\Prob{{\mathbf{P}}}
\def\prob{{\mathbf{P}}}
\def\ProbT{P^\bbT}
\def\probT{P^\bbT}

\def\E{{\mathbf{E}}}
\def\ET{E^\bbT}
\def\last{{\mathrm{last}}}
\def\cF{{\mathcal{F}}}
\def\fF{{\mathfrak{F}}}
\def\fG{{\mathcal{G}}}
\def\cA{{\mathcal{A}}}
\def\cG{{\mathcal{G}}}
\def\cC{{\mathcal{C}}}
\def\cT{{\mathcal{T}}}
\def\WSF{{\mathsf{WSF}}}
\def\bp{{\mathbf{p}}}

\begin{document}

\title{{\bf Mean-field avalanche size exponent for sandpiles on
Galton-Watson trees}}

\author{Antal A.~J\'arai$^{\textup{(a)}}$, Wioletta M. Ruszel$^{\textup{(b)}}$ and Ellen Saada$^{\textup{(c)}}$
\\ {\small $^{\textup{(a)}}$
Department of Mathematical Sciences,}\\
{\small University of Bath,}\\
{\small Claverton Down, Bath, BA2 7AY, United Kingdom}
\\ {\small $^{\textup{(b)}}$
Mathematical Institute,}\\
{\small Utrecht University,}\\
{\small Budapestlaan 6, 3584 CD Utrecht, The Netherlands}
\\
{\small $^{\textup{(c)}}$}
{\small CNRS, UMR 8145, Laboratoire MAP5,}\\
{\small
Universit\'{e} Paris Descartes,} \\
{\small 45, rue des Saints P\`{e}res 75270 Paris Cedex 06, France}\\
}

\maketitle

\maketitle

\begin{abstract}
We show that in Abelian sandpiles on infinite Galton-Watson trees,
the probability that the total avalanche has more than $t$ topplings decays
as $t^{-1/2}$. We prove both quenched and annealed bounds, under
suitable moment conditions. Our proofs are based on an analysis
of the conductance martingale of Morris (2003), that was previously
used by Lyons, Morris and Schramm (2008) to study uniform spanning
forests on $\Zd$, $d\geq 3$, and other transient graphs.
\noindent
\textbf{Subclass:} 60K35, 82C20
\end{abstract}

\section{Introduction and results}

The Abelian sandpile model was introduced  in 1988 by
Bak, Tang and Wiesenfeld in \cite{BTW88}
as a toy model displaying self-organized criticality.
A self-organized critical model is
postulated to drive itself into a critical state
which is characterized by power-law behaviour of,
for example, correlation functions, without fine-tuning
an external  parameter. For a general
overview we refer to \cite{J14, R06} and to some of
the physics literature \cite{D90, D06}. There are
connections of the sandpile model to Tutte polynomials \cite{CB03},
logarithmic conformal
invariance \cite{R13}, uniform spanning trees \cite{D90},
and neuronal communication \cite{BP03}.

Consider a finite connected graph $G=(V\cup \{s\},E)$ with a
distinguished vertex $s$ called
the sink. Assign to each vertex $x\in V$ a natural number
$\eta_x \in \mathbb{N}$ representing
its {\it height} or {\it mass}.

The Abelian sandpile model is defined as follows: choose
at every discrete time step a vertex
$x\in V$ uniformly at random and add {\it mass} 1 to it.
If the resulting mass at $x$ is at least the number of neighbours of $x$,
then we \textit{topple} the vertex  $x$
by sending unit mass to each neighbour of $x$.
Mass can leave the system via the sink $s$, according to a rule depending on the graph.
 The topplings  in $V$ 
will continue until all the vertices
in $V$ are \textit{stable}, that is, they have mass which is smaller
than the number of neighbours.
The sequence of consecutive topplings is called an \textit{avalanche}.
The order of topplings does not matter,
hence the model is called Abelian.
The unique stationary measure for this Markov chain is the
uniform measure on the recurrent configurations.

There are various interesting quantities studied,
for example the avalanche size or diameter
distribution depending on the underlying graph
\cite{BHJ17,DM90, H19, JL93}, the toppling durations,
infinite-volume models \cite{AJ04,MRS02},
and continuous height analogues \cite{JRS15}.

In particular, it is known that on a regular tree (Bethe lattice)
the probability that an avalanche of size at least $t$ occurs,
decays like a power law with mean-field exponent $-1/2$
for large $t$ \cite{DM90}, and the same is true on the
complete graph \cite{JL93}. Very recently, this has been
extended by Hutchcroft \cite{H19} to a large class of graphs
that are, in a suitable sense, high-dimensional. No assumptions
of transitivity are needed in \cite{H19}, but the proofs require
bounded degree. In particular, \cite{H19} shows that
the exponent $-1/2$ holds for the lattice $\mathbb{Z}^d$ for $d\geq 5$,
and also for bounded degree non-amenable graphs. See also
\cite{BHJ17} for related upper and lower bounds on critical
exponents on $\mathbb{Z}^d$ for $d \ge 2$.

In \cite{RRS12} sandpile models on random binomial (resp. binary)
trees are considered, i.e. every vertex has two descendants with
probability $p^2$, one with probability $2p(1-p)$ and none with
probability $(1-p)^2$ (resp. 2 offspring with probability $p$ and
none with probability $1-p$); there, in a toppling, mass 3 is
ejected by the toppling site, independently of its number of
neighbours; hence there is \textit{dissipation}  (that is,
there is mass which is not sent to a neighbouring site, but which is lost) 
when this number is less than 2. It is proven in \cite{RRS12} 
that in a small supercritical regime $p>1/2$ the (quenched and annealed)
avalanche sizes decay exponentially, hence the model is not
critical. Moreover  (see \cite{RRS17}) 
the critical branching parameter for these models is $p=1$.
The reason is that as soon as there exist vertices with degree
strictly less than $2$, the extra dissipation
thus introduced to the system is destroying the criticality of the model. \\

In this paper, we consider an Abelian sandpile model on a
supercritical Galton-Watson branching tree $\bbT$ with
possibly unbounded offspring distribution $\bp=\{p_k\}_{k\geq 0}$
under some moment assumptions. We prove that the probability that
the total avalanche has more than $t$ topplings decays
as $t^{-1/2}$. 
Our proofs rely on a quantitative analysis of the
\textit{conductance martingale} of Morris \cite{LMS08, M03},
that he introduced
to study uniform spanning forests on $\mathbb{Z}^d$ and other
transient graphs). The use  
of this martingale  is the major novelty of our paper, 
and our hope is that this gives insight into
the behaviour of this martingale on more general graphs. \\

Our methods are very different
from those of \cite{H19}.   While
the results of \cite{H19} are stated for bounded degree graphs
(and more generally for networks with vertex conductances bounded
away from $0$ and infinity), Hutchcroft's approach can also be applied
to unbounded degree graphs: In our context, under suitable moment
conditions, the proof methods of \cite{H19} would yield the
$t^{-1/2}$ behaviour with an extra power of $\log t$ present 
\cite{HPC}. \\

We write $\nu_{\T}$ for the probability
distribution of the sandpile model conditioned on the environment $\bbT$.
Let $S$ denote the total number of topplings upon addition at the root,
which is a.s. finite (see later on for details). Then we prove the following.
\begin{theorem}
\label{thm:intro-qn}
Conditioned on the event that $\bbT$ survives, there exists $C = C(\bp)$  such that
for all $t $ large enough depending on $\bbT$ we have
\eqnst
{  \nu_{\T} [ S > t ]
  \le C \, t^{-1/2}. }
Furthermore if $\bp$ has an exponential moment then there exists
$c_0 = c_0(\bbT)$ that is a.s.~positive
on the event that $\bbT$ survives, such that we have
\eqnst
{ \nu_{\T} \big[ S > t \big]
  \ge  c_0 \, t^{-1/2}. }
\end{theorem}
We also have the following annealed bounds.
\begin{theorem}
\label{thm:intro-an}
 Let $\prob$ denote the probability distribution for the
 Galton-Watson trees, and $\E$  the corresponding expectation.
There exists $C = C( \bp ) > 0$ such that
\eqnst
{ \E \big[ \nu_{\T} [ S > t ] \,\big|\, \text{$\bbT$ survives} \big]
  \le C \, t^{-1/2}. }
and if $\bp$ has exponential moment then there exists $c = c( \bp )$ such that
\eqnst
{   \E \big[  \nu_{\T} [ S > t ] \,\big|\, \text{$\bbT$ survives} \big]
  \ge c \, t^{-1/2}. }
\end{theorem}
%

The paper is organized as follows. First in section \ref{sec:notation} we introduce the 
setting and notation and in particular we recall the decomposition of avalanches into waves. 
 In Section \ref{sec:wave-ub} we  prove upper bounds on the waves and in the 
subsequent Section \ref{sec:wave-lb} corresponding lower bounds. 
We deduce the corresponding bounds on $S$ from the bounds on the waves in 
Section \ref{sec:waves} and finally we prove annealed bounds in Section \ref{sec:annealed}.

\section{Notation and preliminaries}\label{sec:notation}
\subsection{Abelian sandpile model on subtrees of the Galton-Watson tree}\label{subsec:GW}
We consider a supercritical Galton-Watson process
with offspring distribution  $\bp=\{p_k\}_{k\ge 0}$ 
with mean $\sum_{k\ge 0} k p_k >1$,
starting with a single individual.

Let us fix a realization $\mathbb{T}(\omega)$ of the family tree of
this Galton-Watson process with root  denoted by  $o$.
We will call 
\begin{equation}\label{eq:F-Tsurvives}
 F :=  \{ \text{$\mathbb{T}$ survives} \},
\end{equation} 
and assume that $\omega \in F$. 
 The random environment $\bbT = \bbT(\omega)$ is defined on a probability space
$(\Omega, \cG, \prob)$.
The edge set of $\mathbb{T}$ is denoted by $E(\mathbb{T})$.
 We use the notation $\mathbb{T}$ to refer to both the tree and to its
vertex set. 
 Take  a subset $A \subset \mathbb{T}$ and let us denote by
 $\partial_E A$ the \textit{edge boundary} of $A$,
i.e. the set of edges $e= (v,u) \in E(\mathbb{T})$ such that
$v \in A$ and $u \in A^c$,
where $A^c$ is the complement of $A$ in $\mathbb{T}$.
We denote by $|A|$ the cardinality of $A$.
 We say that $A$ is \textit{connected} if
the subgraph induced in $\bbT$ is connected.
Then the
 \textit{distance} $d(u, v)$
 between the two vertices $u,v \in A$ is defined as  the number of edges
 of the shortest path joining them within $A$.
 For $v\in \mathbb{T}$ we write $|v|=d(o, v)$.
 The \textit{ (outer) vertex boundary} $\partial_V A $
 is defined as follows. A vertex  $v\in \mathbb{T}$
 belongs to $\partial_V A$ if $v\in A^c$ and there exists
 $u\in A$ such that $(u,v) \in E(\bbT)$.
 Let $\partial^{in}_V A = \{ v\in A: \exists \, w\in A^c \, \text{such that} \, (v,w) \in E(\bbT)\}$
 be the \textit{internal vertex boundary} of $A$.
 We will further use the notation $(V, o)$ for a graph with vertex set $V$ and root $o$.

By a result of 
Chen \& Peres (\cite[Corollary 1.3]{CP04})  
we know that conditioned on $F$ the tree $\bbT$ satisfies
\textit{anchored isoperimetry},  meaning that the edge boundary of a set containing a 
fixed vertex is larger than some positive constant times the volume. This isoperimetric 
inequality ensures an exponential growth condition on the random tree. 

They proved  (case \textit{(ii)} in the proof of
\cite[Corollary 1.3]{CP04}) that there
exists $\delta_0 = \delta_0(\bp) > 0$
and a random variable $N_1 = N_1(\bbT)$ 
that is a.s.~finite on $F$, 
such that for any finite connected $o \in A \subset \bbT$
with $|A| \ge N_1$ we have
\eqn{e:CP}
{ | \partial_E A | \ge \delta_0 |A|. }
It also follows from the proof of \cite[Corollary 1.3]{CP04} that
there exists $c_1 = c_1(\bp) > 0$ such that
\eqn{e:CP-tail}
{ \Prob [ N_1 \ge n \,|\, F ]
  \le e^{- c_1 n}, \quad n \ge 0. }

We denote by $\bbT_{k}=\{ v\in \bbT: d(o,v)=k\}$ 
(respectively $\bbT_{<k} =\{ v\in \bbT:d(o,v) < k \}$)  the set of vertices
at  precisely distance $k$ (respectively at distance less than $k$) from the root, 
and analogously we define $\bbT_{\leq k}$. We write $\bbT(v)$
for the subtree of $\bbT$ rooted at $v$.  For a vertex $v\in  \bbT $
 we denote by $\deg(v)$  the \textit{degree}  $\deg_\bbT(v)$ 
 of  vertex $v$ within $\bbT$ (i.e.
the number of edges in $E(\bbT)$ with one end equal to $v$),
and we denote by $\deg^+(v)$
the forward degree  $\deg^+_\bbT(v)$  of $v$,
that is the number of children of $v$.

For some finite connected subset $H \subset \bbT$ such that $o \in H$
we write $\bbT^*_H$ for the finite connected wired graph,
i.e. such that each vertex in $H^c$ is identified with some cemetery vertex $s$,
called a sink. For a vertex $v\in H $
 we denote by $\deg_H(v)$  the \textit{degree} of  vertex $v$ within $H$ (i.e.
the number of edges in $E(\bbT^*_H)$ with one end equal to $v$),
and we denote by $\deg^+_H(v)$
the forward degree of $v$ within $H$.
We fix such a subset $H$ from now on.

 We gather in the following subsections  results we need
 on the Abelian sandpile model, for which we refer for instance
to \cite{D90,H08,J14,R06}.

\subsubsection{Height configurations and legal topplings}\label{subsec:height-legal}
Height configurations on $\bbT^*_H$ are elements
$\eta\in \{0,1,2,\cdots\}^{H}$.
For $u\in H$, 
$\eta_u$ denotes the height at vertex $u$.
A height configuration $\eta$
is \textit{stable} if $\eta_u\in \{0, 1, 2, \dots, \deg_H(u)-1\}$ for all $u\in H$.
Stable configurations are collected in the set
$\Omega_H$. Note that  $\deg_H(u)$, $u \in  H$,  and $\Omega_{H}$,  depend
on the realization of the Galton-Watson tree $\mathbb{T}$, hence are random.

For a configuration $\eta$, we define the \textit{toppling
operator} $T_u$ via
\[
\left(T_u (\eta)\right)_v=\eta_v-\Delta^{H}_{uv}
\]
where $\Delta^{H}$ is the \textit{toppling matrix}, indexed by vertices $u,v\in H$ and
defined by
\begin{equation*}\label{top}
\Delta^{H}_{uv}=
\begin{cases}
  \deg_H(u),  &\mbox{if}\ u=v \\
  -1,  &\mbox{if}\ (u,v)\in  E(\bbT^*_H).
\end{cases}
\end{equation*}
In words, in a toppling at $u$, $\deg_H(u)$  particles are removed from $u$, and
every neighbour of $u$ receives one particle. Note that $\Delta^{H}$ depends
on the realization of $\bbT$ which hence is random in contrast to the case
of the binary tree studied in \cite{RRS12}.
Therefore there is no dissipation in a toppling, except for the particles
received by the sink of $\bbT^*_H$.

A toppling at $u\in H$ in configuration $\eta$  is called \textit{legal} if
$\eta_u \ge \deg_H(u)$.
A sequence of legal topplings is a composition
$T_{u_n}\circ\cdots\circ T_{u_1} (\eta)$ such that
for all $k=1,\cdots,n$ the toppling at $u_k$ is legal
in $T_{u_{k-1}}\circ\cdots\circ T_{u_1} (\eta)$. The \textit{stabilization}
of a configuration $\eta$ is defined
as the unique stable configuration $\mathrm{S}(\eta)\in\Omega_H$
that arises from $\eta$ by a sequence of legal topplings.
Every $\eta\in \{0,1,2,\cdots\}^{H}$ can be stabilized
thanks to the presence of a sink.
\subsubsection{Addition operator and Markovian dynamics}\label{subs:adop}
Let $u\in H$, the \textit{addition operator} is the map
$a_u:\Omega_H\to\Omega_H$ defined via
\begin{equation*}
a_u \eta = \mathrm{S} (\eta+\delta_u)
\end{equation*}
where $\delta_u\in\{0,1\}^{H}$ is such that
$\delta_u(u)=1$ and $\delta_u(z)=0$ for $z\in H,z\not=u$.
In other words, $a_u\eta$ is the effect of an addition
of a single grain at $u$ in $\eta$,
followed by stabilization.

The dynamics of the sandpile model can be defined as a
discrete-time Markov chain $\{\eta(n), n\in \mathbb{N}\}$
on $\Omega_H$  with 
\begin{equation}\label{markov}
\eta(n)= \prod_{i=1}^n a_{X_i} \eta(0)
\end{equation}
where $X_i, 1\le i\le n$, are i.i.d.\ uniformly chosen vertices in $H$.
\subsubsection{Recurrent configurations, spanning trees and
stationary measure}\label{subsec:rec-stat}
The  set of \textit{recurrent} configurations  $\mathrm{R}_H$
of the sandpile model corresponds to the recurrent states of
the Markov chain \eqref{markov} defined above.
 This  Markov chain has a unique stationary probability measure
$\nu_H$  which is the uniform measure on the set $\mathrm{R}_H$.
There is a bijection between $\mathrm{R}_H$ and the spanning
trees of $\bbT^*_H$ \cite{MD92}, that is useful in analyzing $\nu_H$.

Let $o \in  H_1 \subset H_2 \subset \cdots \subset H_n \subset \cdots$
be a sequence of  finite sets with union equal to $\mathbb{T}$. The
\emph{sandpile measure $\nu_{\T}$ on $\mathbb{T}$} is defined
as the weak limit of the stationary measures $\nu_{H_n}$ for
the sandpile model on $\bbT^*_{H_n}$, when the limit exists.
By \cite[Theorem 3]{JW12}, an infinite volume
sandpile measure $\nu_{\T}$  on $\mathbb{T}$ exists if
each tree in the $\WSF$ (Wired Uniform Spanning
Forest)  on $\mathbb{T}$ has one end almost surely.
The $\WSF$ is defined as the weak limit of the uniform
spanning trees  measure  on $\bbT^*_{H_n}$, as $n \to \infty$.
We refer to \cite{LP16} for background on wired spanning forests.
We define the related measure $\WSF_o$ in the following way.
Identify $o$ and $s$ in $\bbT^*_{H_n}$  and let  $\WSF_o$ be
the weak limit of the uniform spanning tree in the resulting graph  $G_n$  as 
 $n \to \infty$. {}From now on, when working on a finite set $H$,
we will abbreviate this procedure by $H\to\bbT$ (or $H$ goes to $\bbT$). 

Let $\fF_o$ denote the  connected
component of $o$ under $\WSF_o$. Almost sure finiteness of
$\fF_o$ is equivalent to one endedness of the component of
$o$ under $\WSF$, see \cite{LMS08}. 
The one end property for trees with bounded degree in the $\WSF$
of Galton-Watson trees was proven by \cite[Theorem 7.2]{AL07}.
In the unbounded case it follows directly by  \cite[Theorem 2.1]{H18}.
Draw a configuration from the measure  $\nu_{\T}$,  add a particle
at $o$ and carry out all possible topplings. By
\cite[Theorem 3.11]{JR08}, one-endedness of the components and
transience of $\bbT$ (for simple random walk) imply that there
will be only finitely many topplings  $\nu_{\T}$-a.s.,
and as a consequence the total number $S$ of topplings  is a.s. finite.
\subsubsection{Waves, avalanches and Wilson's method}\label{subs:WandA}
Given a stable height configuration $\eta$ and $o\in H$,
we define the \textit{avalanche cluster} $\Av_H(\eta)$
induced by addition at $o$ in $\eta$ to be
the set of vertices in $H$ that have to be toppled
at least once in the course of the stabilization
of $\eta+\delta_o$. Avalanches can be decomposed
into waves  (see \cite{IKP94}, \cite{JR08})
corresponding to carrying out topplings in a special order.
The \textit{first wave} denotes the set of vertices in $H$
which have to be toppled in course of stabilization until $o$
has to be toppled again.  The \textit{second wave} starts again
from $o$ and collects all the vertices involved in the toppling
procedure until $o$ has to be toppled for the second time etc.

Let $N_H(\eta)$ denote the number of waves caused by addition
at $o$ to the configuration $\eta$ in $H$. For fixed $\bbT$,
the avalanche can be decomposed into
\begin{equation*}\label{dec-avalanche}
\Av_H(\eta) = \bigcup_{i=1}^{N_H(\eta)}  W_H^i(\eta)
\end{equation*}
where $W_H^i(\eta)$ is the $i$-th wave. We write
$W^{\last}_H(\eta)$ for $W^{N_H(\eta)}_H(\eta)$.
 Further we denote by
\begin{equation}\label{def:S}
S_H(\eta)= |W_H^1(\eta)| + \cdots+ |W_H^{\last}(\eta)|
\end{equation}
the total number of topplings in  the  avalanche $\Av_H(\eta)$.

Note that waves can be defined on the full tree $\bbT$
as well where now it is possible to have infinitely many waves.
However, due to the almost sure finiteness of the avalanche,
$N_H$ under $\nu_H$ converges weakly to $N$
under the sandpile measure which is  $\nu_{\T}$ -a.s. finite.
Furthermore  $W_H^i$ converges weakly to $W^i$.
We thus have
\begin{eqnarray*}\label{dec-avalanche-infini}
\Av(\eta) &=& \bigcup_{i=1}^{N(\eta)}  W^i(\eta)\\
\label{def:S-infini}
S(\eta)&=& |W^1(\eta)| + \cdots+ |W^{\last}(\eta)| \\\label{SgeAv}
S(\eta)&\ge & | \Av(\eta) |.
\end{eqnarray*}
\begin{lemma}\label{lem:waves}
For any stable configuration $\eta$ on $\bbT$ we have the following.
\begin{enumerate}
\item[(i)] $W^1(\eta)$ equals the connected component of $o$ in
$\{ v\in \bbT: \eta_v =\deg(v)-1\}$ (possibly empty);
\item[(ii)] $N(\eta) = 1 +  \max\{k \in \mathbb{N}: \bbT_k \subset W^1(\eta) \}$,
with the right hand side interpreted as $0$ when $W^1(\eta) = \es$; 
\item[(iii)] $W^1 (\eta) \supset \dots  \supset W^{\mathrm{last}} (\eta)$.
\end{enumerate}
\end{lemma}
\begin{proof}
(i)  Call $A$  the connected component of $o$ in 
$\{ v\in \bbT: \eta_v =\deg(v)-1\}.$ Then all of the vertices in $A$
topple in the first wave (and they topple exactly once).  
On the other hand each vertex in
$\partial_V A$ only receives one particle
and hence will not topple.  \\
(ii) After the first wave vertices  other than $o$ 
in $\partial^{in}_V W^1(\eta)$ have at most $\deg(v)-2$ particles 
and hence $W^2(\eta)$ equals the connected component of $o$ in 
$W^1(\eta) \setminus \partial^{in}_V W^1(\eta)$.
Let us call  $K=\max\{k \in \mathbb{N}: \bbT_k \subset W^1(\eta) \}$.
Then $\bbT_{\leq K} \subset W^1(\eta)$ but there exists $v\in \bbT_K$
such that $v \in \partial^{in}_V W^1(\eta)$ and therefore
$\bbT_{\leq K-1} \subset W^2(\eta)$ but $v\notin W^2(\eta)$.
The claim follows now by  repeating this argument for
$\partial^{in}_V W^2(\eta), W^3(\eta)$, etc. up to $W^{\mathrm{last}} (\eta)$.  \\
(iii) This last assertion follows from the arguments in the proof of (ii).
\end{proof}

Recall that $\bbT$ is a fixed realization of a supercritical Galton-Watson tree.
Observe that in the supercritical case, a.s.~on $F$ there exists
a vertex $v^* = v^*(\bbT)$ such that $v^*$ has at least two children
with an infinite line of descent, and $v^*$ is the closest such vertex
to $o$. Hence, in the sequel we may assume
without loss of generality that our sample $\bbT$ is such that $v^*$ exists.

\begin{lemma}
\label{lem:waves2}
For $\nu_\bbT$-a.e.~$\eta$ there is at most one wave with the
property that $v^*$ topples but one of its children does not.
When this happens, we have $N(\eta) \ge |v^*|+1$, and the wave
in question is $W^{N-|v^*|}(\eta)$.
\end{lemma}

\begin{proof}
Let $o = u_0, \dots, u_{|v^*|} = v^*$ be the path from $o$ to $v^*$. Then for each
$0 \le k \le |v^*|-1$, the only child of $u_k$ with an infinite line of descent is
$u_{k+1}$. This implies that the graph $H_0 := \bbT \setminus \bbT(v^*)$ is finite.
Consider any finite subtree $H$ of $\bbT$ that contains $\{ v^* \} \cup H_0$.
By the burning test of Dhar \cite{D90,H08}, under $\nu_H$ we have
$\eta(w) = \deg(w)-1$ for all $w \in H_0$. Taking the weak limit,
this also holds under $\nu_\bbT$ (which exists for a.e.~$\bbT$).
It follows from this and  Lemma \ref{lem:waves}  
that either $v^*$ does not topple
in the avalanche (when $\eta(v^*) \le \deg(v^*)-2$), or if $v^*$ topples, then there
is an earliest wave $W^\ell(\eta)$  such that $v^*$ topples in  $W^\ell(\eta)$,
 but one of its
children does not. It follows then by induction that in  $W^{\ell+k}(\eta)$ the vertex
$u_{|v^*|-k}$ topples, but $u_{|v^*|-k+1}$ does not, for $1 \le k \le |v^*|$.
Hence $\ell + |v^*| = N$, and the claim follows.
\end{proof}

In addition to the above lemmas, we will use the following upper bound.
 Let $G^\bbT(x,y) = (\Delta^\bbT)^{-1}(x,y)$, where $\Delta^\bbT$
is the graph Laplacian of $\bbT$. This is the same as the Green's
function of the continuous time simple random walk on $\bbT$ that
crosses each edge at rate $1$. 
\begin{lemma}\label{lemma:SP-WSF}
For $\eta$ sampled from $\nu_{\T}$ and the corresponding $\WSF_o$-measure we have
\[
 \nu_{\T} (W^1(\eta) \in \mathcal{A} ) \leq G^{\bbT}(o,o) \WSF_o(\fF_o \in \mathcal{A})
\]
where $\mathcal{A}$ is a cylinder event.
\end{lemma}
\begin{proof}
We first show the statement in finite volume $H$ and then take the weak limit.
Let $\overline{\mathrm{R}}_H$ be the set of configurations that appear
just before a wave (thus each $\overline{\eta}$ satisfies 
$\overline{\eta}(o) = \deg_H(o)$), and write $W_H(\overline{\eta})$ 
for the set of vertices that topple in the wave represented by $\overline{\eta}$.
By \cite{IKP94} there is a bijection between $\overline{\mathrm{R}}_H$
and 2-component spanning forest on $\bbT^*_H$ such that $o$ and $s$
are in different components. Alternatively these are spanning trees
of the graph  $G$ obtained from $\bbT^*_H$ by identifying $o$ and $s$.
Let us call the uniform spanning tree measure on this finite graph
$\WSF_{o,H}$. We have
\begin{equation*}
\begin{split}
 \nu_H(W_H^1(\eta) \in \mathcal{A}) 
 &= \frac{\left |\{\eta\in \mathrm{R}_H: W_H^1(\eta) \in \mathcal{A} \}\right |}{|\mathrm{R}_H|} \\
 &\leq \frac{|\overline{\mathrm{R}}_H|}{|\mathrm{R}_H|}\cdot 
       \frac{\left| \{ \overline{\eta} \in \overline{\mathrm{R}}_H: 
       W_H(\overline{\eta}) \in \mathcal{A} \} \right|}{|\overline{\mathrm{R}}_H|} \\
 &= \mathbb{E}_{\nu_H}(N) \, \WSF_{o,H}(\fF_o\in \mathcal{A})
\end{split}
\end{equation*}
where the last step follows from the bijection.
By Dhar's formula \cite{D90} and taking the weak limit  $H\to\bbT$
(see Subsection \ref{subsec:rec-stat})  we conclude the claim.
\end{proof}

Occasionally, we will use Wilson's algorithm \cite{W96}, that
provides a way to sample uniform spanning trees in finite graphs,
and as such can be used to sample $\fF_o$ under $\WSF_{o,H}$, as follows.
Enumerate $H \setminus \{ o \}$ as $\{ v_1, \dots, v_{|H|-1} \}$.
Run a loop-erased random walk (LERW) in $\bbT^*_H$ from $v_1$ until it hits
$\{ o, s \}$, which yields a path $\gamma_1$. Then run a LERW from $v_2$ until
it hits $\gamma_1 \cup \{ o, s \}$, yielding a path $\gamma_2$, etc.
The union of all the LERWs is a two component spanning forest with $o$ and $s$
in different components, and the component containing $o$ is distributed as
$\fF_o$. By passing to the limit $H \rightarrow \bbT$ and using transience of $\bbT$,
one obtains the following algorithm to sample $\fF_o$ under $\WSF_o$.
Enumerate $\bbT \setminus \{ o \} = \{ v_1, v_2, \dots \}$. Run a LERW
from $v_1$, stopped if it hits $o$, yielding a path $\gamma_1$. Then run a
LERW from $v_2$, stopped, if it hits $\gamma_1 \cup \{ o \}$, yielding a path
$\gamma_2$, etc. Then the union of the paths that attach to $o$ is distributed
as $\fF_o$ under $\WSF_o$. (Compare \cite[Section 10.1]{LP16} on Wilson's
method rooted at infinity.)
\subsection{Electrical networks and the conductance martingale}\label{subsec:elnet}
\subsubsection{Effective conductances and resistances}\label{subsubsec:effcond}
A general reference for this section is the book \cite{LP16}. Let $G=(V,E)$
be a finite or locally finite infinite graph, for example $\bbT^*_H$ or
$\bbT(v)$. We can regard them as an electrical network where each edge
has conductance (and hence resistance) 1.
An \textit{oriented edge} $e=(e^-,e^+)$ (or $e^{\rightarrow}$)
has a head $e^+$ and a tail $e^-$. The set of oriented edges is denoted
by $E^{\rightarrow}$.
In a finite network, the \textit{effective resistance}
$\caR$ between two sets $A$ and $B$
will be denoted by $\caR(A \leftrightarrow B)$.
The \textit{effective conductance} $\caC$ between $A$ and $B$ is equal to
\begin{equation*}
 \caC(A \leftrightarrow B) = \frac{1}{\caR(A \leftrightarrow B)}.
\end{equation*}
In an infinite network $G$, we will need the effective resistance to infinity $\caR(A \leftrightarrow \infty; \, G)$
and
\begin{equation*}
 \caR(A \leftrightarrow \infty; \,G) = \frac{1}{\caC(A \leftrightarrow \infty; \, G)}.
\end{equation*}
 where $\caC(A \leftrightarrow \infty; \, G)$ denotes the effective
conductance to infinity in $G$. 
 \\

Since we are dealing with trees, we will often be able to
compute resistances and conductances using series and parallel laws.
If $G$ is  a finite network  and ${\bf T}$ is the uniform spanning
tree of $G$ we can write
\begin{equation*}\label{eq:res}
\mathbb{P}(e\in {\bf{T}}) = {\mathcal{R}}(e^-\leftrightarrow e^+)
\end{equation*}
due to Kirchhoff's law  \cite{K47}. For any vertex $v \in \bbT$
denote
\begin{equation}\label{eq:defC}
  \caC(v) 
  := \caC(v\leftrightarrow \infty; \, \bbT(v))
  \le \deg^+(v), 
\end{equation}
where the inequality follows since each edge has unit resistance.

The following lemma is a special case of a computation
in the proof of the martingale property
in \cite[Theorem 6]{M03}. For convenience of the reader,
we give here a short proof based on Wilson's
algorithm, which is possible, since we are dealing with trees.

\begin{lemma}\label{l:cond-alt}
Let $o\in A \subset \bbT$ be connected, $B \subset \partial_V A$ and
$e=(e^-,e^+) \in \partial_E A$ with $e^+\notin B$. Then we have
\[
\WSF_{o}(e^+\in \fF_o | A\subset \fF_o, B\cap \fF_o=\emptyset) = \frac{1}{1+\caC(e^+)}.
\]
\end{lemma}
\begin{proof}
Take $H$ large enough such that $A\cup B \cup \{e^+\} \subset H$
and let $G$ be the graph obtained from $\bbT^*_H$ by identifying
$o$ and $s$. Let $\bbT^*_H(e^+)$ be the subgraph
of $\bbT^*_H$ induced by the vertices in $(\bbT(e^+) \cap H) \cup \{s\}$.
Using Wilson's algorithm to sample $\WSF_{o,H}$, we have that
$\WSF_{o,H}(e^+\in \fF_o | A \subset \fF_o, B\cap \fF_o=\emptyset)$
equals the probability that a simple random walk in $\bbT^*_H$
started at $e^+$ hits $e^-$ before hitting $s$.
This equals $[1+\caC(e^+ \leftrightarrow s; \bbT^*_H(e^+))]^{-1}$,
and letting $H$ go to $\bbT$ we obtain the result.
\end{proof}
\subsubsection{The conductance martingale}\label{subs:cond-mart}
Let us fix an environment $\bbT$, and let $\mathfrak{F}$
denote a sample from the measure $\WSF_o$ defined on the
graph $\bbT$. Recall $\fF_o$ is the connected component
of $o$ in $\fF$.

We inductively construct a random increasing sequence
$E_0 \subset E_1 \subset E_2 \subset \dots$
of edges. Put $E_0 = \es$. Assuming $n \ge 0$ and that
$E_n$ has been defined, let $S_n$ be the set of vertices
in the connected component of $o$ in $E_n \cap \mathfrak{F}$
(we have $S_0 = \{ o \}$). Let us call all edges in
$\mathbb{T} \setminus E_n$ that are incident to $S_n$
\emph{active at time $n$}, and let us denote by $\cA_n$ the event
that this set of active edges is empty. On the event $\cA_n$,
that is, when all edges in $\mathbb{T}$ incident to
$S_n$ belong to $E_n$, we set $E_{n+1} = E_n$.
On the event $\cA_n^c$, we select an active edge
$e_{n+1}$, and we set $E_{n+1} = E_n \cup \{ e_{n+1} \}$.
(Note: at this point we have not yet specified how we select
an active edge. In some cases this will not matter, in some
other cases we will make a more specific choice later, 
 see Section \ref{sec:wave-ub}). Note that 
the event $\{ |\fF_o| < \infty \}$ equals
$\bigcup_{n \geq 1} \cA_n$. Let
\begin{equation*}\label{def:mart}
 M_n  := \mathcal{C} (S_n \conn \infty ; \, \mathbb{T} \setminus E_n).
  \end{equation*}
Let $\mathcal{F}_n$ denote the $\sigma$-field
generated by $E_n$ and $E_n \cap \fF$. By
a result of Morris (see  \cite[Theorem 8]{M03}
and \cite[Lemma 3.3]{LMS08}) $M_n$ is an
$\mathcal{F}_n$-martingale.

Since we are dealing with trees, the increments of $M_n$ can be expressed
very simply. Let $\cC_n := \cC(e^+_{n+1})$ (cf. \eqref{eq:defC})
and recall that this is the conductance
from $e^+_{n+1}$ to infinity in the subtree  $\bbT(e^+_{n+1})$.
 Then by Lemma \ref{l:cond-alt} the probability, given $\cF_n$, that $e_{n+1}$
belongs to $\fF_o$ equals $(1 + \cC_n)^{-1}$. On this event,
we have
\eqnst
{ M_{n+1} - M_n
  = - \frac{1}{1 + \frac{1}{\cC_n}} + \cC_n
  = - \frac{\cC_n}{1 + \cC_n} + \cC_n
  = \frac{\cC_n^2}{1 + \cC_n}. }
Here the negative term is the conductance from $e^-_{n+1}$ to infinity
via the edge $e_{n+1}$. 

This implies that conditionally on $\cF_n$ we have
\[
 M_{n+1} - M_n
  = \begin{cases}
    \displaystyle{\frac{\cC_n^2}{1 + \cC_n}} & 
    \text{with probability $\displaystyle{\frac{1}{1 + \cC_n}}$;} \\ \\
   \displaystyle{ -\frac{\cC_n}{1 + \cC_n}} & 
   \text{with probability $\displaystyle{\frac{\cC_n}{1 + \cC_n}}$.}
    \end{cases} 
\]
Let
\eqn{e:Di}
{ D_i
  = \ET  \big[ M_{i+1}^2 - M_i^2 \,\big|\, \cF_i \big]
  = \cC_i \, \frac{\cC_i^2}{(1 + \cC_i)^2}. }
We will use the short notation $\ProbT$ instead of $\WSF_o$ 
from now on and denote $\ET$ the associated expectation.

\section{Upper bound on waves}\label{sec:wave-ub}
In this section we give upper bounds on waves for general offspring
distributions, conditioning the environment on the event $F$ (cf. \eqref{eq:F-Tsurvives}). 

Let $\mathbb{T}'$ denote the subtree of $\mathbb{T}$ consisting 
of those vertices $v$ such that $\mathbb{T}(v)$ is infinite. 
We will write 
\begin{equation}\label{eq:defoverC}
\overline{\cC}(v) := \max \{ \cC(v),\, 1 \}.
\end{equation} 
Recall the random variable $N_1(\bbT)$ from \eqref{e:CP}.
\begin{theorem}
\label{thm:noK}
Suppose that $1<\sum_{k \ge 0} k p_k \leq \infty$.
There exist $C_1 = C_1(\bp)$ and $t_0 = t_0 (\bp)$ 
such that on the event of survival we have
\eqnst
{ \probT \big[ |\fF_o| > t \big]
  \le C_1 \, \cC(o) \, t^{-1/2}, \quad t \ge \max \{ t_0(\bp),\, N_1(\bbT) \}. }
Therefore,
\[ 
\probT \big[ |\fF_o| > t \big]
  \le C_1 \, N_1^{1/2} \, \overline{\cC}(o) \, t^{-1/2},
  \quad t > 0. 
\]

\end{theorem}
We will use the following stopping times:
\eqnsplst
{ \tau^-
  &= \inf \{ n \ge 0 : M_n = 0 \} \\
  \tau_{b,t}
  & = \inf \{ n \ge 0 : M_n \ge b t^{1/2} \}, \quad b > 0,\, t > 0. }
We impose the following restriction on selecting edges to
examine for the martingale. If there is an active
edge $e$ available with
$\cC(e^+)^2 / (1 + \cC(e^+)) < (1/2) t^{1/2}$, we select one
such edge to examine, otherwise we select any other edge.

Observe that on the event $F$, we have $M_0 > 0$ 
 (recall that $M_0=\cC(o)$),  and Doob's inequality gives
\eqnst
{ \probT [ \tau_{1/4,t} < \tau^- ]
   \le \probT \left[ \sup_{n} M_n \ge \frac{1}{4} t^{1/2} \right]
  \le 4 M_0 t^{-1/2}. }
  Moreover, as long as $n<\tau^-$, we have $M_n > 0$. 
Consider the stopping time
\eqnst
{ \sigma
  = \tau_{1/4, t} \wedge \inf \left\{ n \ge 0 :
    \frac{\cC(e^+)^2}{1 + \cC(e^+)} \ge \frac{1}{2} t^{1/2} \text{ for all
    active $e$ at time $n$} \right\}. }
When there are no active edges at all, that is, at time $\tau^-$,
the condition on them holds vacuously, and hence
$\sigma \le \tau^- \wedge \tau_{1/4, t}$.

\begin{lemma}
\label{claim}
On the event $\{ \sigma < \tau^- \}$, 

(i) we have $M_\sigma \le t^{1/2}$;

(ii) we either
have the event $\{ \tau_{1/4, t} < \tau^- \}$ or else
no edges are added to the cluster after time $\sigma$,
that is:
$\fF_\sigma = \fF_n = \fF_{\tau^-}$ for all
$\sigma \le n \le \tau^-$. 
\end{lemma}

\begin{proof}
(i) The claim amounts to showing that
when $M_\sigma \ge \frac{1}{4} t^{1/2}$, we have
$M_\sigma \le t^{1/2}$  (if $M_\sigma \le \frac{1}{4} t^{1/2}$, 
then $M_\sigma \le t^{1/2}$). 
Let $e$ be the edge examined 
at time $\sigma-1$. Then
\eqnst
{ M_\sigma
  \le M_{\sigma-1} + \frac{\cC(e^+)^2}{1 + \cC(e^+)}
  \le \frac{1}{4} t^{1/2} + \frac{1}{2} t^{1/2}
  < t^{1/2}. }
  
(ii) Let us assume that $M_\sigma < \frac{1}{4} t^{1/2}$ (otherwise
the event $\{ \tau_{1/4, t} < \tau^- \}$ has occurred).
Let $e_1, \dots, e_\ell$ be the available edges at time
$\sigma$.
Examine each of the edges $e_1, \dots, e_\ell$ in turn, to determine
whether they belong to $\fF_o$ or not. Suppose that for
some $1 \le j \le \ell$ we have that
$e_j$ is found to belong to $\fF_o$, and let $j$ be the minimal such
index. Then  (recall the definition of $\sigma$) 
\eqnst
{ M_{\sigma+j}
  = M_{\sigma+j-1} + \frac{\cC(e_j^+)^2}{1 + \cC(e_j^+)}
  > \frac{\cC(e_j^+)^2}{1 + \cC(e_j^+)}
  \ge \frac{1}{2} t^{1/2}
  > \frac{1}{4} t^{1/2}. }
Thus the event $\{ \tau_{1/4, t} < \tau^- \}$ occurs.
This proves our claim.
\end{proof}

We have
\eqnsplst
{ M_0^2
  = \ET \left[ M_{\sigma}^2 \, \mathbf{1}_{\sigma < \tau^-} \right]
    - \ET \left[ \sum_{i=0}^{\sigma-1} D_i \right]. }
Here, due to  Lemma \ref{claim}(i),  the first term is bounded above by
\eqnst
{\ET \left[ M_{\sigma}^2 \, \mathbf{1}_{\sigma < \tau^-} \right] \leq  t^{1/2} \ET \left[ M_{\sigma} \, \mathbf{1}_{\sigma < \tau^-} \right]
  = t^{1/2} M_0, }
and hence
\eqnspl{e:D-exp-bnd}
{ \ET \left[ \sum_{i=0}^{\sigma-1} D_i \right]
  \le M_0 t^{1/2}. }

The idea is to show that there cannot be many active edges at time
$\sigma$ from which the conductance is low, and hence there are
sufficiently many terms $D_i$ such that $D_i > c$ for some $c > 0$.

Recall the anchored isoperimetry equation \eqref{e:CP} and exponential bound \eqref{e:CP-tail}.
The following proposition gives a bound on the probability of there
being \emph{any} connected subset of the Galton-Watson tree
that has `many' boundary edges with low conductance to infinity.
Let $o \in A \subset \bbT$ be a  finite  connected set of vertices
such that $|A| = n$. Let us call $e \in \partial_E A$ is
\emph{$\delta$-good} if $\cC(e^+)/(1 + \cC(e^+)) \ge \delta$.
Let us say that \emph{$A$ is $\delta$-good}, if
\eqnsplst
{ \big| \left\{ e \in \partial_E A :
  \text{$e$ is $\delta$-good} \right\} \big|
  \ge \delta \left| \partial_E A \right|. }

We are going to need the \textit{isoperimetric profile function}
(see \cite[Section 6.8]{LP16}) given by:
\eqn{e:psiAt}
{ \psi(A,t)
  := \inf \left\{ | \partial_E K | : A \subset K,\, \text{$K/A$ connected},\,
     t \le |K|_{\deg} < \infty \right\}, }
where $|K|_{\deg} = \sum_{v \in K} \deg(v)$.

\begin{proposition}
\label{prop:x0-good}
Assume $1 < \sum_{k \ge 0} k p_k \leq \infty$.
There exists $\delta_1 = \delta_1(\bp) > 0$ such that
all  finite  connected sets $A$ with
     $o \in A \subset \bbT$ and $|A| \ge N_1$ are $\delta_1$-good.
\end{proposition}

\begin{proof}
Observe that if $o \in A$ and $A$ is connected, then any $K$ inside the
infimum in \eqref{e:psiAt} is a tree, and hence
\eqnst
{ |K|_{\deg}
  = \sum_{v \in K} \deg(v)
  = 2 |K| - 2 + |\partial_E K|. }
This implies that if $|A| \ge N_1(\mathbb{T})$, we have
\eqnst
{ \frac{| \partial_E K |}{ |K|_{\deg} }
  = \frac{| \partial_E K |}{2|K| - 2 + |\partial_E K|}
  \ge \frac{| \partial_E K |}{2|K| + |\partial_E K|}
  \ge \frac{\delta_0 |K|}{2 |K| + \delta_0 |K|}
  = \frac{\delta_0}{2 + \delta_0}. }
Consequently,
\eqnst
{ \psi(A,t)
  \ge \frac{\delta_0}{2 + \delta_0} t
  =: f(t). }
Therefore, an application of \cite[Theorem 6.41]{LP16}   (which gives an upper bound of the effective resistance in terms of integrals over the lower bound of the isoperimetric profile function)  yields that
\eqnst
{ \caR (A \conn \infty)
  \le \int_{|A|_{\deg}}^\infty \frac{16}{f(t)^2} \, dt
  = \frac{16 \, (2 + \delta_0)^2}{\delta_0^2} \, |A|_{\deg}^{-1}. }
Hence
\eqnst
{ \caC (A \conn \infty)
  \ge \frac{\delta_0^2}{16 \, (2 + \delta_0)^2} \, |A|_{\deg}
  \ge \frac{\delta_0^2}{16 \, (2 + \delta_0)^2} \, | \partial_E A |. }
Put 
\begin{equation*}\label{def:delta1}
\delta_1 = \frac{1}{2}\left( \frac{\delta_0^2}{16 \, (2 + \delta_0)^2} \right).
\end{equation*} 
Since
\eqnst
{ \caC (A \conn \infty)
  = \sum_{e \in \partial_E A} \frac{\caC(e^+)}{1 + \caC(e^+)}, }
we have that
\eqnst
{ \left| \left\{ e \in \partial_E A : \frac{\caC(e^+)}{1 + \caC(e^+)} \ge \delta_1 \right\} \right|
  \ge \delta_1 | \partial_E A |, }
which is the claimed inequality.
\end{proof}

\begin{proof}[Theorem \ref{thm:noK}]
Recall the positive constant $\delta_0$ from \eqref{e:CP},
the positive constant $\delta_1$ of Proposition \ref{prop:x0-good},
and the a.s.~finite random variable $N_1 = N_1(\bbT)$ of \eqref{e:CP-tail}.

Assume that $\bbT$ satisfies the event $\{ N_1(\bbT) \le t \}$.
On the event
\eqnst
{ \left\{ \sup_n M_n \le \frac{1}{4} t^{1/2} \right\} \cap
     \big\{ \left| (\text{edges in $\fF_o$}) \right| > t-1 \big\}, }
we have  $|\fF_o| \ge N_1$.
Hence by the anchored isoperimetry equation \eqref{e:CP}
and by Proposition \ref{prop:x0-good} we have
\eqnspl{e:CP+cond}
{ &\left| \text{(edges in $\fF_o$)} \right|
  = | \fF_o | - 1 \\
  &\qquad \le \frac{1}{\delta_0} | \partial_E \fF_o | \\
  &\qquad \le \frac{1}{\delta_0 \, \delta_1} \left| \text{(edges  $e$  in $\partial_E \fF_o$
      with
       $\frac{\cC(e^+)}{1 + \cC(e^+)} \ge \delta_1$}) \right|  \\
  &\qquad \le \frac{1}{\delta_0 \, \delta_1} \left( \frac{1}{\delta_1^3} \, \sum_{i=0}^{\sigma-1} D_i
      + \left| \text{(edges in $\partial_E \fF_o$ examined
      after time $\sigma-1$)} \right| \right), }
where the last inequality used that 
when $\cC_i / (1 + \cC_i) \ge \delta_1$, we have 
(recall \eqref{e:Di})
\eqnst
{ D_i
  = \cC_i \, \frac{\cC_i^2}{(1 + \cC_i)^2}
  \ge \delta_1^3. }
In order to estimate the last term in the right hand side of
\eqref{e:CP+cond}, we use that if $e_1, \dots, e_\ell$ are the
edges that are examined after time $\sigma$, then on the
event $\{ \sup_n M_n < (1/4) t^{1/2} \}$, we have
\eqnst
{ (1/4) t^{1/2}
  > M_\sigma
  = \sum_{j = 1}^\ell \frac{\cC(e_j^+)}{1 + \cC(e_j^+)}
  \ge \ell \, \frac{(1/2) t^{1/2}}{1 + (1/2) t^{1/2}}
  = \ell \, \frac{1}{1 + 2 t^{-1/2}}
  \ge \ell \, (1 - 2 t^{-1/2}), }
and hence for $t \ge 16$ we have
\eqnst
{ \ell
  \le \frac{(1/4) t^{1/2}}{1 - 2 t^{-1/2}}
  \le (1/2) t^{1/2}. }
This gives that the right hand side of \eqref{e:CP+cond} is
at most
\eqnst
{ \frac{1}{\delta_0 \, \delta_1^4} \sum_{i=0}^{\sigma-1} D_i
      + \frac{1}{\delta_0 \, \delta_1} \, \frac{t^{1/2}}{2}. }
The inequality \eqref{e:D-exp-bnd} implies that
\eqnsplst
{ \probT \left[ \sum_{i=0}^{\sigma-1} D_i > \frac{t \, \delta_0 \, \delta_1^4}{2} \right]
  \le \frac{2 \, M_0}{\delta_0 \, \delta_1^4} \, t^{-1/2}. }

Therefore, if $t \ge t_0 := (\delta_0 \, \delta_1)^{-2}$ and $\bbT$
satisfies the event $\{ N_1(\bbT) \le t \}$, we have
$\frac{1}{\delta_0 \, \delta_1} \, \frac{t^{1/2}}{2} \le \frac{t}{2}$, and
hence for all $t \ge t_0$ we have
\begin{eqnarray*}
 &\probT \big[ | \fF_o | > t \big]
  = \probT [ \text{\#(edges in $\fF_o$)} > t-1 ] \\
  &\qquad \le \probT \left[ \sup_n M_n \geq \frac{1}{4} t^{1/2} \right]
      + \probT \left[  \sup_n M_n < \frac{1}{4} t^{1/2} ,\frac{1}{\delta_0 \, \delta_1^4} \,
      \sum_{i=0}^{\sigma-1} D_i > \frac{t}{2} \right] \\
  &\qquad \le 4 M_0 t^{-1/2} + \frac{2 \, M_0}{\delta_0 \, \delta_1^4} \, t^{-1/2} \\
  &\qquad = \cC(o) \, \left[ 4 + \frac{2}{\delta_0 \, \delta_1^4} \right] \, t^{-1/2}. 
\end{eqnarray*}

This completes the proof of the first statement, 
for $t \ge \max \{ t_0(\bp),\, N_1(\bbT) \}$.
The second statement of the theorem follows immediately, 
since $C_1>1$, and              
also $N_1^{1/2} t^{-1/2} > 1$ if $t < N_1$.
\end{proof}

\section{Lower bound on waves}\label{sec:wave-lb}

In this section we prove the lower bound corresponding to Theorem \ref{thm:noK}. 
Denote by $f$ the generating function of $\bp$, that is $f(z) = \sum_{k \ge 0} p_k z^k$.
We introduce the following assumption on $f$:
\eqn{e:ass-beta}
{ \text{there exists $z_0 := e^{\beta_0} > 1$ such that
  $f(z_0) < \infty$.}
  \tag{M-$\beta$} }

\begin{theorem}
\label{thm:lb-beta-noK}
Suppose that $\bp$ satisfies Assumption
\eqref{e:ass-beta} with some $\beta_0 > 0$, and
suppose that $\sum_{k \ge 0} k p_k > 1$.
Then  conditioned on $F$  there exists $c = c(\bbT) > 0$ such that
\eqnst
{ \probT \big[ |\fF_o| > t \big]
  \ge c t^{-1/2}. }
\end{theorem}

We will need the following a.s.~upper bound on the vertex boundary of sets.

\begin{proposition}
\label{prop:LD-conduct}
Under Assumption \eqref{e:ass-beta}, 
there exists an a.s.~finite $C' = C'(\bbT)$, such that for any finite
connected set $o \in A \subset \bbT$ we have
\eqn{e:LD-bdry}
{ |A \cup \partial_V A|
  \le C' |A|. }
\end{proposition}

\begin{proof}
Fix a plane tree $A$ (i.e.~$A$ is a rooted tree with root $o$
and the children of each vertex of $A$ are ordered).
Also fix numbers $n_v, m_v$ for $v \in A$, with the following
properties:
\eqnsplst
{ n_v 
  &= \text{number of children of $v$ in $A$} \\
  n
  &:= |A| 
  = \sum_{v \in A } n_v + 1 \\
  m_v 
  &\ge 0 \\
  d_v 
  &:= n_v + m_v \\
  M
  &:= \sum_{v \in A} m_v. }
For each $v \in A$, fix a subset $I_v \subset \{ 1, \dots, d_v \}$ 
such that $| I_v | = n_v$. 
If we view $A$ as a subtree of $\bbT$ then every vertex $v\in A$ has  forward  degree $n_v$ in $A$ 
and forward  degree $d_v$ in $\T$. Thus each $v\in A$ has $m_v$ children in $\bbT$ which belong to 
$\partial_V A$. We define the event
\eqnsplst
{ E(A, \{ m_v \}, \{ I_v \})
  &= \left\{ \parbox{8cm}{$(\bbT,o)$ has a rooted subtree $(A',o)$ isomorphic to 
    $(A,o)$ such that the forward  degree in $\bbT$ of each $v \in A'$ equals $d_v$
    and the set of children in $A'$ of each $v \in A'$
    equals $I_v$} \right\}. }

The probability of $E(A, \{ m_v \}, \{ I_v \})$ equals
\eqnst
{ \Prob \big[ E(A, \{ m_v \}, \{ I_v \}) \big]
  = \prod_{v \in A} p(d_v)
  = \prod_{v \in A} p(n_v + m_v), }
where for readability we wrote $p(d_v)$ and $p(n_v + m_v)$ instead of 
$p_{d_v}$ and $p_{n_v + m_v}$. Hence, if $1 < e^{\beta} < z_0$, we have
\eqnspl{e:EA-bnd}
{ \Prob \left[ E(A, \{ m_v \}, \{ I_v \}) \right]
  = \exp ( - \beta M ) \,
      \prod_{v \in A} p(n_v + m_v) \,
      e^{\beta \, m_v}. }

Let 
\eqnsplst
{ E'(A, \{ m_v \})
  &= \left\{ \parbox{8cm}{$(\bbT,o)$ has a rooted subtree $(A',o)$ isomorphic to 
    $(A,o)$ such that the  forward  degree in $\bbT$ of each $v \in A'$ equals $d_v$} 
    \right\}. }
Taking a union bound in \eqref{e:EA-bnd} and summing over $\{ I_v \}$ yields:
\eqnspl{e:E'A-bnd}
{ &\Prob \left[ E'(A, \{ m_v \}) \right] \\
  &\qquad \le \exp ( - \beta M ) \,
      \prod_{v \in A} \binom{n_v + m_v}{n_v} \, p(n_v + m_v) \,
      e^{\beta \, m_v} \\
  &\qquad = \exp ( - \beta M ) \,
      \prod_{v \in A} \frac{1}{n_v!} \,
      (m_v + n_v) \, \cdots \, (m_v + 1) \, p(n_v + m_v) \,
      e^{\beta \, m_v}. }
In order to sum over $m_v$, we are going to use that 
\[
 \sum_{m \ge 0} \binom{n+m}{n} \, p(n+m) \, z^m
  = \frac{1}{n!} \sum_{m \ge 0} 
    p(n+m) \, (m+n) \, \cdots \, (m+1) \, z^m 
  = \frac{1}{n!} f^{(n)}(z). 
\]
For a fixed $\widetilde{M}$, let us define
\eqnsplst
{ E''(A, \widetilde{M})
  &= \left\{ \parbox{7cm}{$(\bbT,o)$ has a rooted subtree $(A',o)$ isomorphic to 
    $(A,o)$ such that $\left| \partial_V A' \right| \ge \widetilde{M}$} 
    \right\}. }

Recall that 
$1 < z_1 := e^{\beta} < z_0$. 
Fix some $C''$ and sum \eqref{e:E'A-bnd} over all $\{ m_v \}$, with 
$M \ge \widetilde{M} := (C''-1) n$.
This gives
\eqnspl{e:M-bndE''-2}
{ \Prob \left[ E''(A,\widetilde{M}) \right]
  \le \exp ( - \beta (C''-1) n ) \,
      \prod_{v \in A} \frac{1}{n_v!} \, f^{(n_v)} (z_1). }
Due to Cauchy's theorem, we have
\eqnst
{ \frac{1}{n_v!} f^{(n_v)} (z_1)
  \le f(z_0) \frac{1}{(z_0 - z_1)^{n_v + 1}}
  \le f(z_0) \, C^{n_v + 1}. }
Substituting this into \eqref{e:M-bndE''-2} and summing over
$A$, while keeping $n$ fixed, yields
\begin{eqnarray*}
 &\Prob \left[ \text{$\exists$ connected set $o \in A \subset \bbT$ with
        $|A| = n$ such that $|\partial_V A| > (C''-1) n$} \right] \\
  &\qquad \le \exp ( - \beta (C''-1) n ) \,
      4^n \, f(z_0)^n \, C^{2 n - 1}. 
\end{eqnarray*}

Here we used that there are $\le 4^n$ non-isomorphic rooted plane trees 
$(A,o)$ of $n$ vertices. (This can be seen by considering the depth-first 
search path of $A$ starting from $o$, which gives an encoding of the tree 
by a simple random walk path of length $2n$.)
If $C''$ is sufficiently large, the estimate in the right hand
side is summable in $n \ge 1$, and hence we have
$|A \cup \partial_V A| \le C'' |A| = C'' n$ for all but finitely
many $n$. Increasing $C''$ to some $C'$ if necessary,
yields the claim \eqref{e:LD-bdry}  on the size of the boundary.
\end{proof}

\begin{lemma}
\label{lem:ub-tau-}
Under Assumption \eqref{e:ass-beta}, there exists an a.s.~finite
$C = C(\bbT)$ such that
\eqnst
{ \ET \big[ \tau^- \wedge t \big]
  \le C t^{1/2}, \quad t \ge 1. }
\end{lemma}

\begin{proof}
Note that the set of edges examined by the conductance martingale
up to time $\tau^-$ equals the edges in $\fF_o$ union the edge
boundary of $\fF_o$. Thus
$\tau^- = |\fF_o|-1 + |\partial_V \fF_o|$.
Using \eqref{e:LD-bdry} of Proposition \ref{prop:LD-conduct}, we have
\eqnst
{ \ProbT \big[ \tau^- \ge s \big]
  \le \ProbT \big[ |\fF_o \cup \partial_V \fF_o| \ge s \big]
  \le \ProbT \big[ |\fF_o| \ge (1/C') s \big]. }
The right hand side is at most $C s^{-1/2}$,
due to Theorem  \ref{thm:noK}. 
Summing over $1 \le s \le t$ proves the claim.
\end{proof}

We need one more proposition for the proof of
Theorem \ref{thm:lb-beta-noK}.

\begin{proposition}
\label{prop:var-bnd}
Under Assumption \eqref{e:ass-beta}, there exists
an a.s.~finite $C = C(\bbT)$ such that
\eqnst
{ \sum_{i = 0}^{\tau^- \wedge t - 1} D_i
  \le C (\tau^- \wedge t). }
\end{proposition}

\begin{proof}
Let $A$ be the connected subgraph of $\bbT$ consisting of
the edges inside $\fF_o$ that have been examined by time
$\tau^- \wedge t$ and found to be in $\fF_o$.
Then $|A| \le \tau^- \wedge t$. For times $i$
such that the edge $e_i = (e_i^-, e_i^+)$ examined at time
$i$ was found to be in $\fF_o$, we use the bound (cf. \eqref{eq:defC}, \eqref{e:Di})
\eqnst
{ D_i
  = \cC_i \frac{\cC_i^2}{(1 + \cC_i)^2}
  \le \cC_i
  \le \deg^+(e_i^+). }
The sum of $D_i$ over such $i$ is hence bounded by
$|A \cup \partial_V A|$. 
We can bound the sum of $D_i$ over the
rest of the times by $| \partial_V (A \cup \partial_V A) |$.
Due to Proposition \ref{prop:LD-conduct}, there exists an
a.s.~finite $C = C(\bbT)$ such that
\eqnst
{ \sum_{i = 0}^{\tau^- \wedge t - 1} D_i
  \le |A \cup \partial_V A| + \sum_{w \in \partial_V A} \cC(w)
  \le C' |A| + (C')^2 |A|
  \le C (\tau^- \wedge t). }
\end{proof}

\begin{proof}[Theorem \ref{thm:lb-beta-noK}]
Recall that on the event $F=\{ \mathbb{T} \text{ survives}\}$ we have that $M_0>0.$
Using Proposition \ref{prop:var-bnd} and Lemma \ref{lem:ub-tau-},
we write
\eqnsplst
{ \ET \left[ M_t^2 \right]
  &= \ET \left[ M_t^2 \, \mathbf{1}_{\tau^- > t} \right]
  = M_0^2 + \ET \left[ \sum_{i = 0}^{\tau^- \wedge t - 1} D_i \right]
  \le M_0^2 + C \, \ET \big[ \tau^- \wedge t \big] \\
  &\le M_0^2 + C \, t^{1/2}
  \le C'' t^{1/2}. }
This gives
\eqnst
{ M_0
  = \ET \big[ M_t \big]
  = \ET \left[ M_t \, \mathbf{1}_{\tau^- > t} \right]
  \le \left( \ET \left[ M_t^2 \right] \right)^{1/2} \,
      \ProbT \big[ \tau^- > t \big]^{1/2}, }
and hence
\eqnst
{ \ProbT \big[ \tau^- > t \big]
  \ge \frac{M_0^2}{C'' t^{1/2}}. }
This gives, using \eqref{e:LD-bdry} of Proposition \ref{prop:LD-conduct},
that
\eqnsplst
{ \ProbT \big[ |\fF_o| \ge t \big]
  &\ge \ProbT \big[ |\fF_o \cup \partial_V \fF_o| \ge C' \, t \big] \\
  &= \ProbT \big[ |\fF_o| - 1 + |\partial_V \fF_o| \ge C' \, t - 1 \big] \\
  &= \ProbT \big[ \tau^- \ge C' \, t - 1 \big] \\
  &\ge c_4 \, t^{-1/2}. }
\end{proof}

\section{From waves to avalanches}
\label{sec:waves}
We use  the following  decomposition of the supercritical branching 
process (see \cite[Section 5.7]{LP16}). 
Recall the definition of the subtree $\bbT'$ of $\bbT$:
for any $v \in \bbT$ such that $\bbT(v)$ is finite,
we remove all vertices of $\bbT(v)$ from $\bbT$, and hence
$\bbT'$ consists of those vertices of $\bbT$ with an infinite 
line of descent. Note that $o\in\bbT'$. 
Let $\{ \widetilde{p}_k \}_{k\geq 0}$ be the offspring distribution of $\bbT$
conditioned on extinction.
Then $\bbT$ can be obtained from $\bbT'$ as follows. 
Let $\{ \widetilde{\bbT}^v : v \in \bbT' \}$ be i.i.d.~family trees 
with offspring distribution $\{ \widetilde{p}_k \}_{k\geq 0}$. 
Identify the root of $\widetilde{\bbT}^v$ with vertex $v$ of $\bbT'$.
Then 
\begin{equation*}
\bbT' \cup \left( \cup_{v \in\bbT'} \widetilde{\bbT}^v \right) 
\stackrel{\mathrm{dist}}{=} \bbT.
\end{equation*}

\begin{lemma}
\label{lem:decomF}
Let $v \in \bbT'$. On the event $v \in \fF_o$, we also have 
$\widetilde{\bbT}^v \subset  \fF_o$.
\end{lemma}

\begin{proof}
Use Wilson's algorithm to generate $\fF_o$ by first starting a random walk
at $v$. If this walk hits $o$, all vertices of $\widetilde{\bbT}^v$
will belong to  $\fF_o$.
\end{proof}

\begin{remark}
Alternatively, it is possible to verify directly that a  recurrent  sandpile configuration 
restricted to any set $\widetilde{\bbT}^v \setminus \{ v \}$ is deterministic, and its height
equals $\deg(w) - 1$ at $w$. Hence if $v$ topples in a wave, all of $\widetilde{\bbT}^v$ topples. 
\end{remark}

\subsection{Quenched lower bound on avalanche size}\label{subsec:quenched}
Recall that given a supercritical Galton-Watson tree $\bbT$, we denoted by
$v^* = v^*(\bbT)$ the closest vertex to $o$ with the property that $v^*$
has at least two children with an infinite line of descent. 
 Let $\bbT'_{k}$ ($\bbT'_{\le k}$, etc.) denote the set of vertices
in the $k$-th generation of $\bbT'$ (in all generations up to
generation $k$, etc.), respectively. 
That is, the smallest
integer $k$ such that $|\bbT'_{k+1}| > 1$ is $|v^*|$.

The following theorem implies the quenched lower bound of Theorem \ref{thm:intro-qn}
stated in the introduction.

\begin{theorem}
\label{thm:av-lb}
Under assumption \eqref{e:ass-beta} and $\mu = \sum_{k \ge 0} k p_k > 1$,
there exists $c_0 = c_0(\bbT)$ that is a.s.~positive
on the event when $\bbT$ survives, such that we have
\eqn{e:av-lb}
{  \nu_{\T}  \big[ S > t \big]  \ge  \nu_{\T}  \left[ \left| W^1(\eta) \right| > t \right]
  \ge  \nu_{\T}  \left[ \left| W^{N-|v^*|}(\eta) \right| > t \right]
  \ge c_0 \, t^{-1/2}. }
\end{theorem}

\begin{proof}
The first inequality follows from  \eqref{def:S-infini} 
and the second one  from Lemma \ref{lem:waves} (iii).
For the third inequality, assume the event that $\bbT$ survives.
Let $y_1, \dots, y_\ell$ be the children of $v^*$ with infinite
line of descent, $\ell \ge 2$. Let $\fG$ be the connected component of
$o$ in $\bbT \setminus \left( \cup_{j=1}^\ell \bbT(y_j) \right)$,
and note that $\fG$ is a finite graph.
We will use Wilson's algorithm to construct an event on which $v^*$ is in $\fF_o$ but
$y_1$ is not. 
Let us use Wilson's algorithm with the walks $S^{(*)}, S^{(1)}, S^{(2)}$
started at $v^*, y_1, y_2$ respectively, in this order.
Consider the event:
\eqnst
{ U
  := \left\{ \text{$S^{(*)}$ hits $o$; 
     $S^{(1)}$ does not hit $v^*$; $S^{(2)}$ hits $v^*$} \right\}. }

On this event $\fF_o$ will correspond to
a wave with the property that $v^*$ topples, but
at least one of its children, namely $y_1$, does not. 
Hence by Lemma \ref{lem:waves2}, this wave is
$W^{N-|v^*|}  (\eta)$. Moreover, we have
\eqnst
{ \fF_o
  \supset \fG \cup \fF^{(2)}_o, }
where $\fF^{(2)}_o$ is distributed as the $\mathsf{WSF}_o$ component of 
$y_2$ in $\bbT(y_2)$.
To complete the proof we note that 
\eqnsplst
{ \nu_{\T} \left[ |W^{N-|v^*|} (\eta) | > t \right]
  &\ge \ProbT \big[ U, \, |\fF^{(2)}_o| > t \big]
  = \ProbT \big[ U \big] \, P^{\bbT(y_2)} \big[ |\fF_o| > t \big] \\
  &\ge c(\bbT) \, c(\bbT(y_2)) \, t^{-1/2}}
where the equality follows from the fact that, conditioned on $U$,  
$\fF^{(2)}_o$ is equal in law to $\fF_o$ on $\bbT(y_2)$. The final lower 
bound follows from the transience of the random walk on $\bbT(y_1)$ 
on the one hand,  and 
on Theorem \ref{thm:lb-beta-noK}  on the other hand.  
\end{proof}

\subsection{Upper bound on avalanche size}\label{subsec:upper}

In this section we prove the following avalanche size bound.

\begin{theorem}
\label{thm:av-ub}
Assume that $1 < \sum_{k \ge 0} k p_k \le \infty$.
There exists $C = C(\bp)$ and on the event $F$ 
an a.s.~finite $N_2 = N_2(\bbT)$ such that
for all $t \ge N_2$ we have
\eqnst
{ \ProbT [ S > t ]
  \le C \, t^{-1/2}. }
\end{theorem}

Recall that $N$ denotes the number of waves. This equals $1$ plus the largest
integer $k$, such that the first wave contains all vertices
in the $k$-th generation of $\bbT$,  see Lemma \ref{lem:waves} (ii).

We use the notation $\probT_v$ for the law of a simple
random walk $\{ S_n \}_{n \ge 0}$ on $\bbT$ with $S_0 = v$.
We denote the hitting time of a set $A$ by
$\xi_A := \inf \{ n \ge 0 : S_n \in A \}$.

\begin{lemma}
\label{lem:num-waves}
We have
\eqnst
{ \nu_{\T}  \big[ N \ge k+1 \big]
  \le G^\bbT(o,o) \, \prod_{e : e^+ \in \bbT'_{k}}
      \frac{1}{1 + \cC(e^+)}, \quad k \ge 0,}
where the empty product for $k = 0$ is interpreted as $1$.

\end{lemma}

\begin{proof}
We can bound from above the probability that the first wave contains
$\bbT_{\le k}$ by $G^\bbT(o,o)$ times the probability
that a typical wave contains it. Thus by Lemma \ref{lemma:SP-WSF}
\eqnst
{ \nu_{\T}  \big[ N \ge k+1 \big]
  \le G^\bbT(o,o) \, \ProbT \big[ \fF_o \supset \bbT_{\le k} \big]
  = G^\bbT(o,o) \, \ProbT \big[ \fF_o \supset \bbT'_{k} \big]. }
In the last step, we used that $\bbT_{\le k} \subset \fF_o$ 
if and only if $\bbT'_{k} \subset \fF_o$. This is implied by 
Lemma \ref{lem:decomF}, since if $\fF_o$ misses a vertex 
$w \in \bbT_{\le k}$, it will also necessarily miss an ancestor 
of $w$ lying in $\bbT'_{\le k}$, and hence will also miss a 
vertex of $\bbT'_k$.
Using Wilson's algorithm and Lemma \ref{l:cond-alt} 
with walks started at vertices in $\bbT'_{k}$,
we get that the probability in the
right hand side is at most

\eqnst
{ \prod_{e: e^+ \in \bbT'_{k}} 
     \probT_{e^+} ( \xi_{e^-} < \infty ) 
  = \prod_{e: e^+ \in \bbT'_{k}} \frac{1}{1 + \cC(e^+)}. }
\end{proof}
We denote by $\bp' = \{ p'_k \}_{k \ge 0}$ the offspring distribution 
of $\bbT'$.

\begin{lemma}
\label{lem:as-bnds}
Assume that $1 < \sum_{k \ge 0} k p_k  \le \infty$. \\
(i) We can find a constant $C_2 = C_2 ( \bp )$,
and on the event $F$ an a.s.~finite $K_1 = K_1 (\bbT') \ge N_1(\bbT')$ 
such that for all $k \ge K_1$ we have
\eqnsplst
{ \max \left\{ N_1(\bbT(w)) : w \in \bbT'_{k} \right\}
  \le C_2 \, \big| \bbT'_{k} \big|. }
Moreover, we have
$\Prob [ K_1 \ge  k  \,|\, F ] \le C \, \exp ( - \delta'_0  k  )$, where
$\delta'_0 = \delta_0(\bp')$ is the isoperimetric expansion constant
of $\bp'$.

(ii) 
We can also find $C_3 = C_3 ( \bp )$ and $c_2 = c_2(\bp) > 0$ 
such that for all $k \ge N_1(\bbT')$ we have
\eqnspl{e:big-expression}
{ (k+1)^{1/2} \, \big| \bbT'_{k} \big| \,
    \left( \sum_{w \in \bbT'_k} \overline{\cC}(w) \right) \,
    \prod_{v \in \bbT'_k}\frac{1}{1 + \cC(v)} 
  \le C_3 \, \exp ( - c_2 k ). }
\end{lemma}

\begin{proof}
(i) Conditioned on $\bbT'_{\le k}$, the trees
$\big\{ \bbT(w) : w \in \bbT'_{k} \big\}$ are independent, 
and the variables $N_1(\bbT(w))$ have an exponential tail, due to
\eqref{e:CP-tail}. Hence we have
\eqnsplst
{ &\Prob \Big[ \max \big\{ N_1(\bbT(w)) : w \in \bbT'_{k} \big\}
    > C_2 \, \big| \bbT'_{k} \big| \Big] \\
  &\qquad = \E \bigg[ \Prob \Big[ \max \big\{ N_1(\bbT(w)) : w \in \bbT'_{k} \big\}
    > C_2 \, \big| \bbT'_{k} \big| \,\Big|\, \bbT_{\le k} \Big] \bigg] \\
  &\qquad \le \E \left[ \sum_{w \in \bbT'_{k}}
    \Prob \Big[ N_1(\bbT(w)) > C_2 \, \left| \bbT'_{k} \right| \,\Big|\,
    \bbT'_{\le k} \Big] \right] \\
  &\qquad \le \E \left[ \left| \bbT'_{k} \right| \, C \,
    \exp \left( - c \, C_2 \, \left| \bbT'_{k} \right| \right) \right]. }

If $C_2 > 2/c$, then the right hand side is at most
\eqnspl{e:exp-bnd}
{ C \, \E [ \exp ( - |\bbT'_{k}| ) ]. }
If $k \ge N_1(\bbT')$, then 
\begin{equation}\label{eq:crucbd}
| \bbT'_{k} | \ge \delta'_0 | \bbT'_{<k} | \ge \delta'_0 k
\end{equation}
 and hence
\eqref{e:exp-bnd}  is summable in $k \ge 1$. Therefore, statement (i) follows from the
Borel-Cantelli Lemma.

(ii) Let us write the sum over $w$, together with the product over $v$ in the form:
\eqnsplst
{ \sum_{w \in \bbT'_k} \frac{\overline{\cC}(w)}{1 + \cC(w)} \, 
    \prod_{v \in \bbT'_k : v \not= w} \frac{1}{1 + \cC(v)}
  \le \sum_{w \in \bbT'_k} \prod_{v \in \bbT'_k : v \not= w} \frac{1}{1 + \cC(v)} }
Assume $k \ge N_1(\bbT')$. Then Proposition \ref{prop:x0-good} can be applied 
with $A = \bbT'_{<k}$ (since $|A| \ge k \ge N_1(\bbT')$), and this gives that, for $\delta'_1 := \delta_1(\bp')$, we have
at least for a proportion $\delta'_1$ of the $\delta'_1$-good vertices $v$
that
\[
\caC(v) > \frac{\caC(v)}{1 + \caC(v)} \ge \delta'_1,
\]
so for these $v$ we have $1/(1 + \caC(v)) < (1 + \delta'_1)^{-1}$. 
Therefore
\eqnst
{ \prod_{v \in \bbT'_k : v \not= w} \frac{1}{1 + \cC(v)}
  \le \left( 1 + \delta'_1 \right)^{-\delta'_1 | \bbT'_k |+1}
  \le (1 + \delta'_1) \, \exp \left( - (\delta'_1)^2 \, | \bbT'_k | \right). }
Thus the left hand side expression in \eqref{e:big-expression} is bounded above by 
\eqnst
{ (k+1)^{1/2} \,  | \bbT'_k |  \, (1 + \delta'_1) \, \exp ( - (\delta'_1)^2 \, |\bbT'_k| ). }
Since $| \bbT'_k | \ge \delta'_0 \, k$, the statement follows.
\end{proof}

In what follows, we write 
\begin{equation}\label{eq:calT}
\cT = \cT(k) := \bbT'_{<k} \cup \left( \cup_{v \in \bbT'_{< k}} \widetilde{\bbT}^v \right).
\end{equation}

\begin{lemma}
\label{lem:cT-bnd}
Assume that $1 < \sum_{k \ge 0} k p_k \le \infty$. 
There exists an a.s.~finite $K_2 = K_2(\bbT') \ge K_1$ such that
for all $k \ge K_2(\bbT')$ we have
\eqn{eq:c}
{ |\cT|
  \le (1/\alpha) \left| \bbT'_{<k} \right|
  \le (\alpha \, \delta_0)^{-1} \, \left| \bbT'_{k} \right|. }
Moreover,
\[
 \Prob [ K_2 \ge  k ]
  \le C \exp ( - c  k ). 
\]
\end{lemma}

\begin{proof}
Note that the size of $\widetilde{\bbT}^o$ has an exponential tail; see for example
\cite[Theorem 13.1]{H63}. Thus there exists 
$\lambda_0 = \lambda_0 (\bp) > 0$ such that  
\eqn{eq:exptildeT}
{ \E \big[ \exp ( \lambda_0 | \widetilde{\bbT}^o | ) \big]
  =: C(\lambda_0) 
  < \infty. }
Let $0 < b \le 1/2$ be a number that we fix with the property that 
\begin{equation}
\label{eq:for-b}
  C(\lambda_0)^b 
  \le e^{\lambda_0/4}.
\end{equation}
Conditionally on $\bbT'_{\le k}$, the trees $\{ \widetilde{\bbT}^v : v \in \bbT'_{< k} \}$
are i.i.d.~with the distribution of $\widetilde{\bbT}^o$.
Therefore, for $\alpha := b/(1+b) \le 1/3$, using \eqref{eq:exptildeT} and \eqref{eq:for-b},
we have that 
\eqnsplst
{ \Prob \left[ |\cT| > (1/\alpha) |\bbT'_{<k}| \right]
  &= \sum_{A} \Prob \left[ \bbT'_{<k} = A \right] \,
    \Prob \big[ |\cT| > (1/\alpha) |A| \,\big|\, \bbT'_{<k} = A \big] \\
  &= \sum_{A} \prob \left[ \bbT'_{<k} = A \right] \,
    \Prob \left[ \sum_{v \in A} \left| \widetilde{\bbT}^v \right| > (1/b) |A| \right] \\
  &\le \sum_{A} \prob \left[ \bbT'_{<k} = A \right] \,
    e^{-\lambda_0 \, (1/b) \, |A|} \, C(\lambda_0)^{|A|} \\
  &\le \sum_{A} \prob \left[ \bbT'_{<k} = A \right] \,
    e^{- (3/4) \, (\lambda_0/b) \, |A|}
  \le e^{- (3/4) \, k \, \lambda_0 / b}. }
Thus the claim follows from the Borel-Cantelli Lemma.
\end{proof}

\begin{proof}[Theorem \ref{thm:av-ub}]
In the course of the proof we are going to choose $\overline{K} = \overline{K}(t) \ge K_2$
(recall $K_2$ defined in Lemma \ref{lem:cT-bnd}). 
We can then write:
\eqnspl{e:three-terms}
{ \ProbT [ S > t ]
  &\le \ProbT [ N \ge \overline{K}+1 ]
      + \ProbT [ 1 \le N \le K_2,\, S > t ]
      + \sum_{K_2 \le k < \overline{K}} \ProbT [ N = k+1,\, S > t ]. }
The first term in the right hand side of \eqref{e:three-terms} can be bounded using 
Lemma \ref{lem:num-waves}:
\eqnst
{ \ProbT [ N \ge \overline{K}+1 ]
  \le G^\bbT(o,o) \, \prod_{e : e^+ \in \bbT'_{\overline{K}}}
      \frac{1}{1 + \cC(e^+)}. }
Since $\overline{K} \ge K_1 \ge N_1(\bbT')$, we can apply Proposition \ref{prop:x0-good}
to $A = \bbT'_{<k}$, and hence
\eqnsplst
{ \ProbT [ N \ge \overline{K}+1 ]
  &\le G^\bbT(o,o) \, \left( 1 + \delta'_1 \right)^{ - \delta'_1 \, |\bbT'_{\overline{K}}|}
  \le G^\bbT(o,o) \, \exp ( - (\delta'_1)^2 \, |\bbT'_{\overline{K}}| ). }

Let us choose 
\[
 \overline{K}
  = \min \left\{ k \ge 0 : \left| \bbT'_k \right| \ge C_3 \log t \right\}, 
\]
where $C_3 = C_3(\bp') := [ 2 (\delta'_1)^2 ]^{-1}$. With this choice, we have
\eqn{e:term1-bound}
{ \ProbT [ N \ge \overline{K}+1 ]
  \le G^\bbT(o,o) \, t^{-1/2}. }

We turn to the second term in the right hand side of \eqref{e:three-terms}. 
Since $S(\eta) = |W^1(\eta)| + \dots + |W^{N}(\eta)|$, where
$W^1(\eta) \supset \dots \supset W^{N}(\eta)$,  see Lemma \ref{lem:waves} (iii), we can write 
\eqnsplst
{ \ProbT [ 1 \le N \le K_2,\, S > t ]
  \le \ProbT [ |W^1(\eta)| > t/K_2 ] 
  \le G^\bbT(o,o) \, \ProbT [ |\fF_o| > t/K_2 ], }
where we used Lemma \ref{lemma:SP-WSF} in the last step.
An application of Theorem \ref{thm:noK} gives:
\eqnspl{e:term2-bnd}
{ \ProbT [ 1 \le N \le K_2,\, S > t ]
  \le C_1 \, G^\bbT(o,o) \, N_1^{1/2} \, \overline{\cC}(o) \, K_2^{1/2} \, t^{-1/2}. }

Finally, we bound the third term in the right hand side of \eqref{e:three-terms}.
Let $K_2 \le k < \overline{K}$.
Using again \eqref{def:S-infini}, Lemma \ref{lem:waves} and Lemma \ref{lemma:SP-WSF}, 
we can write
\eqnspl{e:middle-term-expand}
{ \ProbT [ N = k+1,\, S > t ]
  &\le \ProbT \Big[ N = k+1, \, |W^1(\eta)| > \frac{t}{k+1} \Big] \\
  &\le \ProbT \Big[ N \ge k+1, \, |W^1(\eta)| > \frac{t}{k+1} \Big] \\
  &= \ProbT \Big[ W^1(\eta) \supset \bbT'_{k},\, |W^1 (\eta)| > \frac{t}{k+1} \Big] \\
  &\le G^\bbT(o,o) \, \ProbT \Big[ \fF_o \supset \bbT'_{k},\, |\fF_o| > \frac{t}{k+1} \Big] \\
  &= G^\bbT(o,o) \, \ProbT \big[ \fF_o \supset \bbT'_{k} \big] \, 
    \probT \Big[ | \fF_o| > \frac{t}{k+1} \,\Big|\, \fF_o \supset \bbT'_k \Big]. } 
An application of Lemma \ref{lem:num-waves} yields that 
\eqn{e:wave-prob-term}
{ \ProbT \big[ \fF_o \supset \bbT'_{k} \big]
  \le \prod_{v \in \bbT'_{k}}
      \frac{1}{1 + \cC(v)}. }
We proceed to bound the conditional probability in the right hand side of
\eqref{e:middle-term-expand}.  For any $w \in \bbT'_k$, let us write
$\fF_{o,w} = \fF_o \cap \bbT(w)$. This way, conditionally on
$\fF_o \supset \bbT'_k$, we have
\eqnst
{ \fF_o
  = \cT \cup \Big( \bigcup_{w \in \bbT'_{k}} \fF_{o,w} \Big), }
where $\mathcal{T}$ was defined in \eqref{eq:calT}, and  where
 the conditional distribution of $\fF_{o,w}$ equals
that of $\fF_o(\bbT(w))$.
Then, using the restriction $k \ge K_2$, we have
\eqnspl{e:k-wave-bound}
{ &\probT \Big[ |\fF_o| > \frac{t}{k+1} \,\Big|\, \fF_o \supset \bbT'_{k} \Big] 
  \le \probT \bigg[ \sum_{w \in \bbT'_{k}} |\fF_{o,w}|
      > \frac{t}{k+1} - |\cT| \,\bigg|\, \fF_o \supset \bbT'_{k} \bigg] \\
  &\qquad \le \probT \bigg[ \sum_{w \in \bbT'_{k}} |\fF_{o,w}|
      > \frac{t}{k+1} - (\alpha \, \delta_0)^{-1} \, |\bbT'_{k}|
      \,\bigg|\, \fF_o \supset \bbT'_{k} \bigg] \\
  &\qquad \le
      \sum_{w \in \bbT'_{k}} P^{ \bbT(w)}
      \left[ |\fF_{o}(\bbT(w))| > \frac{t}{(k+1) \,  |\bbT'_{k}|}
      - (\alpha \, \delta_0)^{-1} \right] \\
  &\qquad \le 
      \sum_{w \in \bbT'_{k}} P^{ \bbT(w)}
      \left[ |\fF_{o}(\bbT(w))| > \frac{t}{2 \, (k+1) \, |\bbT'_{k}|} \right]. }  
In the second inequality we use \eqref{eq:c} and 
in the last step we use that on the one hand $k < \overline{K}$ implies
$|\bbT'_k| < C_3  \, \log t$ and on the other hand 
$k \le | \bbT'_{<k} | \le (\delta'_0)^{-1} |\bbT'_k| \le (\delta'_0)^{-1} \,  C_3  \, \log t$
(cf. \eqref{eq:crucbd}) and hence the inequality follows for $t \ge t_1 = t_1(\bp)$.

Applying Theorem \ref{thm:noK} to the probability in the
right hand side of \eqref{e:k-wave-bound} yields the upper
bound
\[
  C' \, t^{-1/2} \, (k+1)^{1/2} \, |\bbT'_{k}|^{1/2} \,
     \sum_{w \in \bbT'_{k}}
     \overline{\cC}(w) \, N_1^{1/2}(\bbT(w)). 
\]
Due to $k \ge K_2 \ge K_1$, and Lemma \ref{lem:as-bnds}(i), this expression
is at most
\eqnspl{e:large-k-bnd}
{ 
  C'' \, t^{-1/2} \, (k+1)^{1/2} \, |\bbT'_{k}| \, \sum_{w \in \bbT'_{k}} \overline{\cC}(w). }  
Substituting  \eqref{e:wave-prob-term} and \eqref{e:large-k-bnd} into the 
right hand side of \eqref{e:middle-term-expand} and using Lemma \ref{lem:as-bnds}(ii)
yields 
\eqnspl{e:term3-bnd}
{ \sum_{K_2 \le k < \overline{K}} \ProbT [ N = k+1,\, S > t ]
  \le C \, t^{-1/2} \, G^\bbT(o,o) \, \sum_{k \ge K_2} \exp ( - c k )
  \le C \, G^\bbT(o,o) \, t^{-1/2}. }
The inequalities \eqref{e:term1-bound}, \eqref{e:term2-bnd} and \eqref{e:term3-bnd} 
substituted into \eqref{e:three-terms} complete the proof of the theorem.
\end{proof}

\section{Annealed bounds}\label{sec:annealed}

Finally, we prove annealed bounds.

\begin{theorem}
\label{thm:annealed}
(i) Under Assumption \eqref{e:ass-beta}, there exists
$c = c( \bp ) > 0$ such that
\eqnst
{ \E \big[ \nu_{\T} [ S > t ] \,\big|\, \text{$\bbT$ survives} \big]
  \ge \E \big[  \nu_{\T}[ |\Av (\eta) | > t ] \,\big|\, \text{$\bbT$ survives} \big]
  \ge c \, t^{-1/2}. }
(ii) Assume that $1 < \sum_{k \ge 0} k p_k \le \infty$.
There exists $C = C( \bp )$ such that
\eqnst
{ \E \big[  \nu_{\T} [ |\Av (\eta) | > t ] \,\big|\, \text{$\bbT$ survives} \big]
  \le \E \big[  \nu_{\T} [ S > t ] \,\big|\, \text{$\bbT$ survives} \big]
  \le C \, t^{-1/2}. }
\end{theorem}

\begin{proof}
Part (i) follows immediately after taking expectations
in \eqref{e:av-lb} of Theorem \ref{thm:av-lb}.

For part (ii), we take expectations in the right hand sides of
\eqref{e:term1-bound}, \eqref{e:term2-bnd} and \eqref{e:term3-bnd}. 
We detail the bound on the expectation of \eqref{e:term2-bnd}, the 
other two are similar and simpler. Recall the notation 
  $\caC(o)$ in \eqref{eq:defC}, and  
$\overline{\cC}(o)$ in \eqref{eq:defoverC}. We similarly denote by 
$\caR(o)$  the effective resistance in $\bbT$ from $o$ to infinity. 
We have $G^\bbT(o,o) = \caR(o)$.
Therefore, $G^\bbT(o,o) \, \overline{\cC}(o) = \max \{ 1,\, \caR(o) \}$, and
we need to bound the expectation of 
\[ 
\max \{ 1, \, \caR(o) \} \, N_1^{1/2} \, K_2^{1/2}. 
\]
Here $N_1$ has an exponential tail, due to \eqref{e:CP-tail}, and 
$K_2$ has an exponential tail due to Lemma \ref{lem:cT-bnd}.
We now show that $\caR(o)$ also has an exponential tail, which immediately
implies that the expectation is finite.

First observe that $\caR(o)$ is also the effective resistance in $\bbT'$
from $o$ to infinity, hence we may restrict to $\bbT'$. Recall that
$\{ p'_k \}_{k \ge 0}$ denotes the offspring distribution of $\bbT'$.
In the case $p'_1 = 0$, there is at least binary branching, and hence
$\caR(o) \le 1$. Henceforth we assume $0 < p'_1 < 1$.
Let $v_\es$ be the first descendant of $o$ in $\bbT'$, where 
the tree branches, that is, there are single offspring until 
$v_\es$,  but $v_\es$ has at least two offspring.  Consider only the 
first two offspring of $v_\es$. Let $v_{1}$ and
$v_{2}$ be the first descendants of $v_\es$ where branching occurs,
that is, each individual on the path between $v_\es$ and $v_i$ has a 
single offspring, but $v_i$ has at least two offspring ($i = 1, 2$).
Analogously, we define
$v_{\eps_1, \dots, \eps_k}$ for each 
$(\eps_1, \dots, \eps_k ) \in \{ 1, 2 \}^k$, $k \ge 0$.

Let $R_\es$ be the resistance between $o$ and $v_\es$ (this is 
the same as the generation difference, since each edge has resistance 1), 
let $R_{v_i}$ be the resistance between $v_{\emptyset}$ and $v_i$ (for $i=1,2$) and 
more generally let $R_{\eps_1, \dots, \eps_k}$ be the resistance between 
$v_{\eps_1, \dots, \eps_{k-1}}$ and $v_{\eps_1, \dots, \eps_k}$
for $k \ge 1$. These random variables are independent, and apart 
from $R_\es$, they are identically distributed with 
distribution 
$\Prob [ R_{\eps_1,\dots,\eps_k} = r ] = (p'_1)^{r-1} (1 - p'_1)$, $r \ge 1$.
The variable $R_\es$ has distribution:
$\Prob [ R_\es = r ] = (p'_1)^{r} (1 - p'_1)$, $r \ge 0$.

For any $0 < t < - \log (p'_1)$ the resistance variables 
all satisfy the bound
\eqnst
{ \E [ \exp ( t R_{\eps_1, \dots, \eps_k} ) ]
  \le \phi(t) 
  := \frac{1 - p'_1}{p'_1} \frac{p'_1 e^{t}}{1 - p'_1 e^t}  
  = \frac{(1-p'_1) e^{t}}{1-p'_1 e^t}. }
We fix $t_0 = - \frac{1}{2} \log (p'_1) > 0$, so that for all $0 < t \le t_0$
the right hand side is bounded above by $(1 + \sqrt{p_1})/\sqrt{p_1} = \sqrt{C'_2} < \infty$.

By the series and parallel laws, the resistance
between $o$ and $\{ v_{1}, v_{2} \}$ is 
\eqnspl{e:resist}
{ R_\es + \frac{1}{\frac{1}{R_{1}} + \frac{1}{R_{2}}}. }
By the inequality between the harmonic mean and arithmetic mean,
\eqref{e:resist} can be bounded above by
\eqnsplst
{ R_\es + \frac{1}{2} \, \frac{1}{\frac{\frac{1}{R_{1}} + \frac{1}{R_{2}}}{2}}
  \le R_\es + \frac{1}{2} \, \frac{R_{1} + R_{2}}{2}
  = R_\es + \frac{R_{1}}{4} + \frac{R_{2}}{4}. } 
Iterating this argument, we get for the effective resistance  $\caR(o)$ 
between $o$ and infinity,  
\eqnsplst
{ \caR(o) \le R_\es + \frac{1}{4} (R_{1} + R_{2})
      + \frac{1}{16} (R_{1,1} + R_{1,2} + R_{2,1} + R_{2,2})
      + \dots. }
Consequently, by Jensen's inequality, we have
\eqnsplst
{ \E [ \exp (t \caR(o)) ]
  &\le \E [ \exp ( t R_1 ) ] \, \E \left [ \exp \left ( \frac{t}{2} R_{1} \right ) \right ] \,
      \E \left[ \exp \left ( \frac{t}{4} R_{1,1} \right ) \right ] \, \dots \\
  &\le \E [ \exp ( t R_1 ) ] \, \E [ \exp ( t R_{1} ) ]^{1/2} \,
      \E [ \exp ( t R_{1,1} ) ]^{1/4} \, \dots \\
  &\le \phi(t)^{1 + \frac{1}{2} + \frac{1}{4} + \dots} \\
  &= \phi(t)^2 
  \le C'_2, \quad 0 < t \le t_0. }
This yields the claimed exponential decay, and the proof is complete.
\end{proof}

\noindent
\textbf{Acknowledgements:}
 The authors thank the following institutions for hospitality and support:
 MAP5 lab. at Universit\'e Paris Descartes, Bath University, Delft University
 as well as Institut Henri Poincar\'e (UMS 5208 CNRS-Sorbonne Universit\'e, endowed with
 LabEx CARMIN, ANR-10-LABX-59-01) - Centre Emile Borel, 
 where  part of this work was done during the trimester 
 ``Stochastic Dynamics Out of Equilibrium''. 
Finally, the authors would like to thank the anonymous referee for 
suggestions that improved the manuscript significantly.

\end{document}